\documentclass[a4paper,11pt]{article}
\usepackage{fancyhdr,fancybox}
\usepackage{indentfirst}
\usepackage{titlesec}
\usepackage{verbatim}
\usepackage[sort&compress, numbers]{natbib}
\usepackage{array,dcolumn,tabularx}
\usepackage{setspace}
\usepackage{geometry}
\usepackage[small]{caption2}
\usepackage{sectsty}
\usepackage{microtype}
\DisableLigatures[f]{encoding = *, family = * }

\usepackage{times}     

\usepackage{latexsym}
\usepackage{amsmath}   
\usepackage{amssymb}   
\usepackage{amsbsy}
\usepackage{amsthm}
\usepackage{amsfonts}
\usepackage{mathrsfs}  
\usepackage{bm}        
\usepackage{relsize}   
\usepackage{hyperref}  

\newcommand{\paperfont}{\fontsize{12pt}{1.3\baselineskip}\selectfont}
\sectionfont{\fontsize{13}{15}\selectfont}
\begin{document}


\theoremstyle{definition}
\makeatletter
\thm@headfont{\bf}
\makeatother
\newtheorem{theorem}{Theorem}[section]
\newtheorem{definition}[theorem]{Definition}
\newtheorem{lemma}[theorem]{Lemma}
\newtheorem{proposition}[theorem]{Proposition}
\newtheorem{corollary}[theorem]{Corollary}
\newtheorem{remark}[theorem]{Remark}
\newtheorem{example}[theorem]{Example}
\newtheorem{assumption}[theorem]{Assumption}
\numberwithin{equation}{section}

\lhead{}
\rhead{}
\lfoot{}
\rfoot{}

\renewcommand{\refname}{References}
\renewcommand{\figurename}{Figure}
\renewcommand{\tablename}{Table}
\renewcommand{\proofname}{Proof}

\newcommand{\dnumiag}{\mathrm{diag}}
\newcommand{\tr}{\mathrm{tr}}
\newcommand{\dnum}{\mathrm{d}}

\newcommand{\Enum}{\mathbb{E}}
\newcommand{\Pnum}{\mathbb{P}}
\newcommand{\Rnum}{\mathbb{R}}
\newcommand{\Cnum}{\mathbb{C}}
\newcommand{\Znum}{\mathbb{Z}}
\newcommand{\Nnum}{\mathbb{N}}
\newcommand{\abs}[1]{\left\vert#1\right\vert}
\newcommand{\set}[1]{\left\{#1\right\}}
\newcommand{\norm}[1]{\left\Vert#1\right\Vert}
\newcommand{\innp}[1]{\langle {#1}]}
\newcommand{\lintg}{\lfloor g(\tau)\rfloor}
\newcommand{\rintg}{\lceil g(\tau)\rceil}
\newcommand{\lint}[1]{\left\lfloor#1\right\rfloor}
\newcommand{\rint}[1]{\left\lceil#1\right\rceil}
\newcommand{\F}{\mathscr{F}}

\title{\textbf{Sharp moderate maximal inequalities for upward skip-free Markov chains}}
\author{Chen Jia$^{1}$\\
\footnotesize $^1$Department of Mathematical Sciences, University of Texas at Dallas, Richardson, TX 75080, U.S.A.\\
\footnotesize E-mail: jiac@utdallas.edu\\}
\date{}                              
\maketitle                           
\thispagestyle{empty}                

\paperfont

\begin{abstract}
The $L^p$ maximal inequalities for martingales are one of the classical results in probability theory. Here we establish the sharp moderate maximal inequalities for upward skip-free Markov chains, which include the $L^p$ maximal inequalities as special cases. Furthermore, we apply our theory to two specific examples and obtain their moderate maximal inequalities: the first one is the M/M/1 queue and the second one is an upward skip-free Markov chain with large death jumps. These two examples have the same total birth and death rates. However, the former exhibits a phase transition phenomenon while the latter does not.\\

\noindent 
\textbf{Keywords}: birth-death process, single-birth process, M/M/1 queue, moderate function, Lenglart domination principle, good $\lambda$ inequality \\

\noindent
\textbf{AMS Subject Classifications}: 60J27, 60E15, 60K25, 60J80
\end{abstract}

\section{Introduction}
A continuous-time Markov chain on the nonnegative integers is called upward skip-free if its generator matrix $Q = (q_{ij})$ satisfies $q_{ij} = 0$ for any $j\geq i+2$. Such kind of Markov chains only allows nearest-neighbor birth jumps but allows large death jumps. They can be viewed as a natural generalization of the classical birth-death processes \cite{karlin1975first}. In the literature, upward skip-free Markov chains are also known as skip-free Markov chains to the right or single-birth processes. Since upward skip-free Markov chains allow large death jumps, they can be used to describe the evolution of the size of a population that undergoes catastrophes \cite{brockwell1982birth, brockwell1986extinction}.

In recent years, upward skip-free Markov chains have drawn increasing attention and there has been a great number of results on various of their properties such as uniqueness, recurrence, ergodicity, exponential ergodicity, strong ergodicity, spectral theory, first passage time, first hitting time, and quasi-limiting distribution \cite{abate1989spectral, kijima1993quasi, chen1999single, zhang2001strong, mao2004exponential, fill2009hitting, chen2014unified}. Another important question that has not been answered yet is how far an upward skip-free Markov chain can travel (how large the population size can become) before a given deterministic or random time. This problem is closely related to the topic of maximal inequalities in probability theory.

In fact, the $L^p$ maximal inequalities for martingales are one of the classical results in probability theory. Let $M = \set{M_t:t\geq 0}$ be a continuous local martingale with $M_0 = 0$. The Burkholder-Davis-Gundy inequality \cite[Page 160]{revuz1999continuous} claims that for any $p>0$, there exist two positive constants $c_p$ and $C_p$ such that for any stopping time $\tau$ of $M$,
\begin{equation*}
c_p\Enum[M]^{p/2}_\tau \leq \Enum[\sup_{0\leq t\leq\tau}|M_t|^p] \leq C_p\Enum[M]^{p/2}_\tau,
\end{equation*}
where $[M] = \set{[M]_t:t\geq 0}$ is the quadratic variation process of $M$. Over the past two decades, significant progress has been made in the maximal inequalities for diffusion processes \cite{graversen1998maximal, graversen1998optimal, graversen2000maximal, peskir2001maximal, peskir2001bounding, yan2004ratio, yan2005lp, yan2005lpestimates, botnikov2006davis, lyulko2014sharp}. In particular, Peskir \cite{graversen2000maximal} has proved the $L^1$ maximal inequalities for a large class of diffusion processes by using the Lenglart domination principle and obtained satisfactory results.

In this paper, we study the moderate maximal inequalities for upward skip-free Markov chains, which include the $L^p$ maximal inequalities as special cases. Our theory is based on two steps. The first step is to establish the $L^p$ maximal inequalities for some $p>0$ by using the discrete version of the Lenglart domination principle, while the second step is to establish the moderate maximal inequalities for any moderate function $F$ by using the discrete version of the good $\lambda$ inequality. Furthermore, we give two examples to illustrate the main results of this paper and discuss the related phase transition phenomenon.

The content of this paper is organized as follows. In Section 2, we give the $L^p$ maximal inequalities for some particular $p>0$ under the so-called Peskir condition. In Section 3, we give the moderate maximal inequalities for general moderate functions under stronger assumptions. In Sections 4 and 5, we apply our abstract theorems to two specific examples: the M/M/1 queue and an upward skip-free Markov chain with large death jumps. Although the total birth and death rates of these two processes are exactly the same, the former exhibits a phase transition phenomenon, while the latter does not due to the large death jumps. The detailed proofs of the main results are given in Sections 6 and 7.

\section{$L^p$ maximal inequalities}
Let $X = \set{X_t:t\geq 0}$ be a homogenous, conservative, and non-explosive continuous-time Markov chain on the nonnegative integers $\Znum^+ = \set{0,1,2,\cdots}$ with generator matrix $Q = (q_{ij})$. Recall that $X$ is called upward skip-free if $q_{ij} = 0$ for any $j\geq i+2$. In the literature, upward skip-free Markov chains are also called skip-free Markov chains to the right or single-birth processes. For simplicity, the total birth and death rates of $X$ are denoted by
\begin{equation*}
\lambda_n = q_{n,n+1},\;\;\;\mu_n = \sum_{j=0}^{n-1}q_{nj},\;\;\;n\geq 0,
\end{equation*}
respectively. It is obvious that the classical birth-death processes \cite{karlin1975first} are special examples of upward skip-free Markov chains. Throughout this paper, we assume that $X_0 = 0$ and $\lambda_n>0$ for any $n\geq 0$.

To proceed, we define two functions $m$ and $f$ on $\Znum^+$ by
\begin{equation*}
m_n = \sum_{i=0}^n\frac{F_{in}}{\lambda_i},\;\;\;f_n = \sum_{k=0}^{n-1}m_k,
\end{equation*}
where $F_{in}$ with $n>i\geq 0$ are defined recursively by
\begin{equation*}
F_{ii} = 1,\;\;\;F_{in} = \frac{1}{\lambda_n}\sum_{k=i}^{n-1}F_{ik}\sum_{j=0}^kq_{nj}.
\end{equation*}
It is easy to see that
\begin{equation*}
f_{n+1}-f_n = m_n \geq \frac{1}{\lambda_n} > 0.
\end{equation*}
Moreover, it has been proved \cite{chen2014unified} that $X$ is nonexplosive if and only if
\begin{equation*}
f_\infty := \lim_{n\rightarrow\infty}f_n = \sum_{n=0}^\infty m_n = \infty.
\end{equation*}
Thus $f$ is a strictly increasing function on $\Znum^+$ with $f_0 = 0$ and $f_\infty = \infty$. For convenience, we extend $f$ to be a strictly increasing continuous function on $\Rnum^+ = [0,\infty)$ and let $g$ be the inverse function of $f$. It is easy to see that $g$ is also a strictly increasing continuous function on $\Rnum^+$ with $g(0) = 0$ and $g(\infty) = \infty$.

Let $X^* = \set{X^*_t:t\geq 0}$ be the maximum process of $X$ defined by
\begin{equation*}
X^*_t=\sup_{0\leq s\leq t}X_s.
\end{equation*}
The following theorem, whose proof can be found in Section \ref{proof1}, gives the $L^p$ maximal inequalities for upward skip-free Markov chains. The condition \eqref{peskir} in the following theorem is the discrete analogue of that proposed by Peskir in the case of diffusion processes \cite{peskir2001bounding}.
\begin{theorem}\label{thm1}
Assume that the following Peskir condition holds for some $p>0$:
\begin{equation}\label{peskir}
\sup_{n\geq 1}\frac{f_n}{n^p}\sum_{k=n+1}^\infty\frac{k^p-(k-1)^p}{f_k}<\infty.
\end{equation}
Then there exist two positive constants $c_p$ and $C_p$ such that for any stopping time $\tau$ of $X$,
\begin{equation*}
c_p\Enum\lintg^p \leq \Enum(X^*_\tau)^p \leq C_p\Enum\rintg^p.
\end{equation*}
Here $\lint{x}$ denotes the largest integer that is smaller than or equal to $x$, $\rint{x}$ denotes the smallest integer that is larger than or equal to $x$, and $\lint{\infty} = \rint{\infty} := \infty$.
\end{theorem}

\begin{remark}
It is easy to check that both $\lintg$ and $\rintg$ do not depend on the continuous extension of $f$.
\end{remark}

\begin{remark}\label{reason}
In general, $\lintg$ in the lower bound and $\rintg$ in the upper bound cannot be replaced by $g(\tau)$ because of the discrete nature of the state space. To see this, we focus on the pure birth process $X$ with birth rates $\lambda_n = \lambda$ for any $n\geq 0$ and the stopping time $\tau$ is chosen as $\tau = \tau_1\wedge t$, where $\tau_1 = \inf\{t\geq 0:X_t = 1\}$. In this case, it is easy to check that $f_n = n/\lambda$ and thus $g(t) = \lambda t$. Moreover, we have
\begin{equation*}
\Enum(X^*_\tau)^p  = \Pnum(\tau_1\leq t) = 1-e^{-\lambda t},
\end{equation*}
and
\begin{equation*}
\Enum g(\tau)^p  = \lambda^p\Enum (\tau_1\wedge t)^p
= \lambda^p\left[\int_0^tu^p\lambda e^{-\lambda u}du+t^pe^{-\lambda t}\right].
\end{equation*}
This suggests that
\begin{equation*}
\lim_{t\rightarrow 0}\frac{\Enum(X^*_\tau)^p}{\Enum g(\tau)^p} =
\begin{cases}
0,\;\;\;0<p<1,\\
\infty,\;\;\;p>1.
\end{cases}
\end{equation*}
Thus when $0<p<1$, it is impossible to find $c_p>0$ such that for any stopping time $\tau$ of $X$,
\begin{equation*}
c_p\Enum g(\tau)^p \leq \Enum(X^*_\tau)^p.
\end{equation*}
Similarly, when $p>1$, it is impossible to find $C_p>0$ such that for any stopping time $\tau$ of $X$,
\begin{equation*}
\Enum(X^*_\tau)^p \leq C_p\Enum g(\tau)^p.
\end{equation*}
\end{remark}

In fact, Theorem \ref{thm1} can be rewritten in a neater form where the upper and lower bounds are given by the same function. As a tradeoff, we need to consider the $L^p$ maximal inequalities of $X+1$ rather than $X$.
\begin{corollary}\label{cor1}
Assume that the Peskir condition \eqref{peskir} holds for some $p>0$. Then there exist two positive constants $c_p$ and $C_p$ such that for any stopping time $\tau$ of $X$,
\begin{equation*}
c_p\Enum(g(\tau)+1)^p \leq \Enum(X^*_\tau+1)^p \leq C_p\Enum(g(\tau)+1)^p.
\end{equation*}
\end{corollary}

\begin{proof}
It is easy to see that for any $p>0$, there exist two positive constants $k_p$ and $K_p$ such that for any $x\geq 0$,
\begin{equation*}
k_p(x^p+1) \leq (x+1)^p \leq K_p(x^p+1).
\end{equation*}
This shows that
\begin{equation*}
\Enum(g(\tau)+1)^p \lesssim \Enum g(\tau)^p+1 \lesssim \Enum\lintg^p+1 \lesssim \Enum(X^*_\tau)^p+1 \lesssim \Enum(X^*_\tau+1)^p,
\end{equation*}
where $x\lesssim y$ means that there exists $C>0$ only depending on $p$ such that $x\leq Cy$. On the other hand, we have
\begin{equation*}
\Enum(X^*_\tau+1)^p \lesssim \Enum(X^*_\tau)^p+1 \lesssim \Enum\rintg^p+1 \lesssim \Enum g(\tau)^p+1 \lesssim \Enum(g(\tau)+1)^p.
\end{equation*}
The above two equations give the desired result.
\end{proof}

\section{Moderate maximal inequalities}
The maximal inequalities in the above section only hold for some special $p>0$ for which the Peskir condition is satisfied. In this section, we shall prove that under stronger assumptions, the maximal inequalities can be established for all $p>0$ and even for more general functions. We first recall the following definition \cite[Page 164]{revuz1999continuous}.
\begin{definition}
A function $F:\Rnum^+\rightarrow\Rnum^+$ is called moderate if\\
(a) it is a continuous increasing function with $F(0) = 0$;\\
(b) there exists $\beta>1$ such that
\begin{equation}\label{requirement}
\sup_{x>0}\frac{F(\beta x)}{F(x)}<\infty.
\end{equation}
\end{definition}

It is easy to see that if $F$ is a moderate function, then \eqref{requirement} holds for any $\beta\geq 1$. In particular, $F(x) = x^p$ is a moderate function for any $p>0$. We next introduce a key definition.
\begin{definition}
$X$ is called controllable if there exist $\beta>1$ and $C,\gamma>0$ such that for any $t\geq 0$ and sufficiently large integer $k$,
\begin{equation}\label{control}
\Pnum_k(X^*_t\geq\lint{\beta k})\leq C\Pnum_0(X^*_t\geq\lint{\gamma k}),
\end{equation}
where $\Pnum_k(\cdot) = \Pnum(\cdot|X_0 = k)$.
\end{definition}

The following theorem, whose proof can be found in Section \ref{proof2}, gives the upper bound of the moderate maximal inequalities for upward skip-free Markov chains.
\begin{theorem}\label{thm2}
Assume that the Peskir condition \eqref{peskir} holds for some $p>0$. If $X$ is controllable, then for any moderate function $F$, there exists $C_F>0$ such that for any stopping time $\tau$ of $X$,
\begin{equation*}
\Enum F(X^*_\tau+1) \leq C_F\Enum F(g(\tau)+1).
\end{equation*}
\end{theorem}

The following theorem, whose proof can be found in Section \ref{proof2}, gives the lower bound of the moderate maximal inequalities for upward skip-free Markov chains.
\begin{theorem}\label{thm3}
Assume that there exist $\beta,M>1$ such that the following condition holds:
\begin{equation}\label{jia}
\lim_{\delta\downarrow 0}\sup_{k\in\Znum,\atop k\geq M}\frac{f(\lint{\delta k})}{f(\lint{\beta k})-f(k)} = 0.
\end{equation}
Then for any moderate function $F$, there exists $c_F>0$ such that for any stopping time $\tau$ of $X$,
\begin{equation*}
\Enum F(X^*_\tau+1) \geq c_F\Enum F(g(\tau)+1).
\end{equation*}
\end{theorem}

It is crucial to note that in the above theorem, the Peskir condition \eqref{peskir} is not a necessary condition for the lower bound. Combining the above two theorems, we obtain the following corollary, which is the main result of this paper.
\begin{corollary}\label{cor2}
Assume that the conditions of Theorems \ref{thm2} and \ref{thm3} are satisfied. Then for any moderate function $F$, there exist two positive constants $c_F$ and $C_F$ such that for any stopping time $\tau$ of $X$,
\begin{equation*}
c_F\Enum F(g(\tau)+1) \leq \Enum F(X^*_\tau+1) \leq C_F\Enum F(g(\tau)+1).
\end{equation*}
In particular, for any $p>0$, there exist two positive constants $c_p$ and $C_p$ such that for any stopping time $\tau$ of $X$,
\begin{equation*}
c_p\Enum(g(\tau)+1)^p \leq \Enum(X^*_\tau+1)^p \leq C_p\Enum(g(\tau)+1)^p.
\end{equation*}
\end{corollary}

\begin{remark}
The conclusion of the above corollary, which holds for any moderate function $F$, is much stronger than that of Corollary \ref{cor1}, which only holds for some particular $p>0$. In the above corollary, $g(\tau)+1$ and $X^*_\tau+1$ cannot be replaced by $g(\tau)$ and $X^*_\tau$ for the same reasons as in Remark \ref{reason}.
\end{remark}

\section{M/M/1 queue and the related phase transition}
In the following two sections, we shall apply the above abstract theorems to two specific examples. In this section, we consider the moderate maximal inequalities for the M/M/1 queue. Recall that $X$ is called an M/M/1 queue if $X$ is a birth-death process whose birth and death rates are given by
\begin{equation*}
\begin{split}
\lambda_n = \lambda,\;\;\;n\geq 0,\;\;\;\mu_n = \mu,\;\;\;n\geq 1.
\end{split}
\end{equation*}
For convenience, let $\alpha = \mu/\lambda$ be the ratio of the death and birth rates.

If $X$ is an M/M/1 queue, then for any $n>i\geq 0$,
\begin{equation*}
F_{in} = \frac{\mu}{\lambda}\sum_{k=i}^{n-1}F_{ik}\delta_{k,n-1} = \alpha F_{i,n-1}.
\end{equation*}
This shows that $F_{in} = \alpha^{n-i}$ for any $n\geq i$ and
\begin{equation*}
m_n = \frac{1}{\lambda}\sum_{i=0}^nF_{in} =
\begin{cases}
\frac{n+1}{\lambda},\;\;\;&\alpha=1,\\
\frac{\alpha^{n+1}-1}{\lambda(\alpha-1)},\;\;\;&\alpha\neq 1.
\end{cases}
\end{equation*}
Thus we have
\begin{equation*}
f_n = \sum_{k=0}^{n-1}m_k =
\begin{cases}
\frac{n(n+1)}{2\lambda},\;\;\;&\alpha=1,\\
\frac{\alpha(\alpha^n-1)-(\alpha-1)n}{\lambda(\alpha-1)^2},\;\;\;&\alpha\neq 1.
\end{cases}
\end{equation*}

Let $h:\Rnum^+\rightarrow\Rnum^+$ be a strictly increasing continuous function with $h(0) = 0$ and $h(\infty) = \infty$ defined by
\begin{equation}\label{hfunction}
h(x) =
\begin{cases}
\frac{x}{\lambda(1-\alpha)},\;\;\;&\alpha<1,\\
\frac{x^2}{2\lambda},\;\;\;&\alpha=1,\\
\frac{\alpha(\alpha^x-1)}{\lambda(\alpha-1)^2},\;\;\;&\alpha>1.
\end{cases}
\end{equation}
It is easy to see that $f_n\sim h_n$ as $n\rightarrow\infty$, that is, $f_n/h_n\rightarrow 1$ as $n\rightarrow\infty$. Thus there exist two positive constants $c$ and $C$ such that $ch_n \leq f_n \leq Ch_n$ for any $n\geq 0$.

\begin{lemma}
Let $X$ be an M/M/1 queue. Then the following three statements hold:\\
(a) If $\alpha<1$, then the Peskir condition \eqref{peskir} holds if and only if $0<p<1$;\\
(b) If $\alpha=1$, then the Peskir condition \eqref{peskir} holds if and only if $0<p<2$;\\
(c) If $\alpha>1$, then the Peskir condition \eqref{peskir} holds for any $p>0$.
\end{lemma}

\begin{proof}
By the mean value theorem, the Peskir condition holds for $p>0$ if and only if
\begin{equation*}
\sup_{n\geq 1}\frac{h_n}{n^p}\sum_{k=n+1}^\infty\frac{k^{p-1}}{h_k}<\infty.
\end{equation*}
When $\alpha<1$, the above power series converges if and only if $0<p<1$. For any $0<p<1$,
\begin{equation*}
\frac{h_n}{n^p}\sum_{k=n+1}^\infty\frac{k^{p-1}}{h_k} = n^{1-p}\sum_{k=n+1}^\infty k^{p-2}
\leq n^{1-p}\int_n^\infty x^{p-2}dx = \frac{1}{1-p}.
\end{equation*}
Thus we have proved (a). Similarly, we can prove (b). When $\alpha>1$, the Peskir condition holds for $p>0$ if and only if
\begin{equation}\label{aim}
\sup_{n\geq 1}\frac{\alpha^n}{n^p}\sum_{k=n+1}^\infty\frac{k^{p-1}}{\alpha^k}<\infty.
\end{equation}
Note that
\begin{equation*}
\frac{\alpha^n}{n^p}\sum_{k=n+1}^\infty\frac{k^{p-1}}{\alpha^k}
= \frac{1}{n^p}\sum_{k=1}^\infty\frac{(k+n)^{p-1}}{\alpha^k}.
\end{equation*}
When $p\leq 1$, we have $(k+n)^{p-1}\leq 1$. When $p>1$, we have
\begin{equation*}
(k+n)^{p-1}\leq \max\set{2^{p-2},1}(k^{p-1}+n^{p-1}).
\end{equation*}
In both cases, it is easy to see that \eqref{aim} holds. Thus we have proved (c).
\end{proof}

\begin{remark}
According to the above lemma, when $\alpha<1$ (or $\alpha = 1$), it is impossible to obtain the $L^p$ maximal inequality for $p\geq 1$ (or $p\geq 2$) via Theorem \ref{thm1}. This shows the limitations of the Peskir condition.
\end{remark}

\begin{lemma}
Let $X$ be an M/M/1 queue. Then the condition \eqref{jia} holds for $\beta = 2$.
\end{lemma}

\begin{proof}
It is easy to check that $f(2k)-f(k)\sim h(2k)-h(k)$ as $k\rightarrow\infty$. Thus when $k$ is sufficiently large, we have
\begin{equation*}
\frac{f(\lint{\delta k})}{f(2k)-f(k)} \leq \frac{2Ch(\lint{\delta k})}{h(2k)-h(k)}
\leq \frac{2Ch(\delta k)}{h(2k)-h(k)}.
\end{equation*}
When $\alpha\leq 1$, it is easy to check that there exists $M\geq 1$ such that
\begin{equation*}
\lim_{\delta\downarrow 0}\sup_{k\geq M}\frac{h(\delta k)}{h(2k)-h(k)} = 0.
\end{equation*}
When $\alpha>1$, we have
\begin{equation*}
\frac{h(\delta k)}{h(2k)-h(k)} = \frac{\alpha^{\delta k}-1}{\alpha^{2k}-\alpha^k} = \frac{y^\delta-1}{y^2-y},
\end{equation*}
where $y = \alpha^k$. When $y$ is sufficiently large, we have $y^2-y\geq y\vee1$ and thus the last term of the above equation is controlled by both $y^{\delta-1}$ and $y^\delta-1$. For any $\epsilon>0$ and $0<\delta<1/2$, when $y\geq 1/\epsilon^2$, we have $y^{\delta-1}\leq\epsilon$. For any $\delta\leq\log_{1/\epsilon^2}(1+\epsilon)$, when $y\leq 1/\epsilon^2$, we have $y^\delta-1\leq\epsilon$. The above analysis shows that there exists $M\geq 1$ such that
\begin{equation*}
\lim_{\delta\downarrow 0}\sup_{y\geq M}\frac{y^\delta-1}{y^2-y} = 0,
\end{equation*}
which gives the desired result.
\end{proof}

The following lemma is interesting in its own right.
\begin{lemma}\label{birthdeath}
Let $X$ be a birth-death process. Assume that $\set{\lambda_n:n\geq 0}$ is a decreasing sequence and $\set {\mu_n:n\geq 1}$ is an increasing sequence. Then for any $m,n\in\mathbb{Z}^+$ and $m\leq n$,
\begin{equation*}
\Pnum_m(X^*_t\geq n)\leq \Pnum_0(X^*_t\geq n-m).
\end{equation*}
In particular, $X$ is controllable.
\end{lemma}

\begin{proof}
Consider a continuous-time Markov chain $(Y,Z) = \set{(Y_t,Z_t):t\geq 0}$ on the state space
\begin{equation*}
S = \set{(i,j)\in\mathbb{Z}^+\times\mathbb{Z}^+:i\geq j}
\end{equation*}
whose all possible nonzero transition rates are given by
\begin{equation}\label{transition1}
\begin{split}
q_{(i,j),(i+1,j+1)} &= \lambda_i,\;\;\;q_{(i,j),(i,j+1)} = \lambda_j-\lambda_i,\\
q_{(i,j),(i-1,j-1)} &= \mu_j,\;\;\;q_{(i,j),(i-1,j)} = \mu_i-\mu_j.
\end{split}
\end{equation}
Assume that $(Y,Z)$ starts from $(m,0)$. It is easy to check that
\begin{equation*}
\begin{split}
\sum_{l=0}^k q_{(i,j),(k,l)} &= q_{ik},\;\;\;\forall 0\leq j\geq i,i\neq k,\\
\sum_{k=l}^\infty q_{(i,j),(k,l)} &= q_{jl},\;\;\;\forall i\geq j,j\neq l.
\end{split}
\end{equation*}
Since the first equation holds for all $0\leq j\geq i$ and the second one holds for all $i\geq j$, it is easy to see that \cite[Chapter III, Lemma 1.2 and Theorem 1.3]{liggett2005interacting}
\begin{equation*}
\begin{split}
\Pnum(Y_{t+h}=k|Y_t=i,\mathscr{F}_t) = q_{ik}h+o(h),\;\;\;\forall i\neq k,\\
\Pnum(Z_{t+h}=l|Z_t=j,\mathscr{F}_t) = q_{jl}h+o(h),\;\;\;\forall j\neq l,
\end{split}
\end{equation*}
where $\mathscr{F}_t = \sigma(Y_s,Z_s,0\leq s\leq t)$. This shows that both $Y$ and $Z$ are birth-death processes with birth rates $\lambda_n$ and death rates $\mu_n$, where $Y$ starts from $m$ and $Z$ starts from 0. According to the transition rates given in \eqref{transition1}, each state $(i,j)$ can jump to $(i+1,j+1)$ or $(i-1,j-1)$. In this case, the distance between $Y$ and $Z$ remains the same. Moreover, each state $(i,j)$ with $i>j$ can also jump to $(i,j+1)$ or $(i-1,j)$. In this case, the distance between $Y$ and $Z$ becomes closer. Therefore, the distance between $Y$ and $Z$ remains the same or becomes closer after every single jump. This shows that
\begin{equation*}
\Pnum_m(Y^*_t\geq n)\leq \Pnum_0(Z^*_t\geq n-m),
\end{equation*}
which gives the desired result.
\end{proof}

The following theorem gives the moderate maximal inequalities for the M/M/1 queue.
\begin{theorem}
Let $X$ be an M/M/1 queue. Then for any moderate function $F$, there exist two positive constants $c_F$ and $C_F$ such that for any stopping time $\tau$ of $X$,
\begin{equation*}
c_F\Enum F(g(\tau)+1) \leq \Enum F(X^*_\tau+1) \leq C_F\Enum F(g(\tau)+1).
\end{equation*}
Moreover, the asymptotic behavior of the function $g$ is given by
\begin{equation*}
g(t) \sim
\begin{cases}
\lambda(1-\alpha)t,\;\;\;&\alpha<1,\\
\sqrt{2\lambda t},\;\;\;&\alpha=1,\;\;\;\textrm{as}\;t\rightarrow\infty.\\
\log_\alpha t,\;\;\;&\alpha>1,
\end{cases}
\end{equation*}
\end{theorem}

\begin{proof}
Combining Corollary \ref{cor2} and the above three lemmas, we obtain the moderate maximal inequalities for $X$. Moreover it is easy to check that $g(t)\sim h^{-1}(t)$ as $t\rightarrow\infty$. The asymptotic behavior of $g$ follows directly from \eqref{hfunction}.
\end{proof}

The above theorem reveals a phase transition of the M/M/1 queue as the parameter $\alpha = \mu/\lambda$ varies. For simplicity, we choose a very large time $t$. When the birth rate is larger than the death rate, $X^*_t$ on average behaves as $\lambda(1-\alpha)t$. When the birth and death rates are equal, $X^*_t$ on average behaves as $\sqrt{2\lambda t}$. When the birth rate is smaller than the death rate, $X^*_t$ on average behaves as $\log_\alpha t$.

\section{Upward skip-free Markov chains with large death jumps}
As a comparison with the M/M/1 queue, we next consider the moderate maximal inequalities for an upward skip-free Markov chain with large death jumps. Specifically, we consider an upward skip-free Markov chain $X$ whose transition rates are given by
\begin{equation}\label{rates}
\begin{split}
\lambda_n &= \lambda,\;\;\;n\geq 0,\\
\mu_n &= q_{n0} = \mu,\;\;\;n\geq 1,\;\;\;q_{nj} = 0,\;\;\;1\leq j<n.
\end{split}
\end{equation}
The total birth and death rates of $X$ are exactly the same as those of the M/M/1 queue. For convenience, let $\alpha = \mu/\lambda$ be the ratio of the total death and birth rates.

It is easy to check that for any $n>i\geq 0$,
\begin{equation*}
F_{in} = \frac{\mu}{\lambda}\sum_{k=i}^{n-1}F_{ik}.
\end{equation*}
This suggests that $F_{i,i+1} = \alpha$ and $F_{i,n+1} = (\alpha+1)F_{in}$ for any $n\geq i+1$. Thus for any $n\geq i+1$,
\begin{equation*}
F_{in} = \alpha(\alpha+1)^{n-i-1}.
\end{equation*}
It is then easy to check that
\begin{equation*}
\begin{split}
m_n &= \frac{1}{\lambda}\sum_{i=0}^nF_{in} = \frac{1}{\lambda}(\alpha+1)^n,\\
f_n &= \sum_{k=0}^{n-1}m_k = \frac{1}{\mu}[(\alpha+1)^n-1].
\end{split}
\end{equation*}
Thus we obtain that
\begin{equation*}
g(t) = \log_{\alpha+1}(\mu t+1).
\end{equation*}

In analogy to the proofs in the above section, it can be proved that the Peskir condition \eqref{peskir} holds for any $p>0$ and the condition \eqref{jia} holds for any $\beta>1$.
\begin{lemma}
Let $X$ be an upward skip-free Markov chain. Assume that $\{\lambda_n:n\geq 0\}$ is a decreasing sequence and $\mu_n = q_{n,0} = \mu$ for any $n\geq 1$. Then for any $m,n\in\mathbb{Z}^+$ and $m\leq n$,
\begin{equation*}
\Pnum_m(X^*_t\geq n)\leq \Pnum_0(X^*_t\geq n-m).
\end{equation*}
In particular, $X$ is controllable.
\end{lemma}

\begin{proof}
Consider a continuous-time Markov chain $(Y,Z) = \set{(Y_t,Z_t):t\geq 0}$ on the state space
\begin{equation*}
S = \set{(i,j)\in\mathbb{Z}^+\times\mathbb{Z}^+:i\geq j}
\end{equation*}
whose all possible nonzero transition rates are given by
\begin{equation}\label{transition2}
\begin{split}
q_{(i,j),(i+1,j+1)} &= \lambda_i,\;\;\;q_{(i,j),(i,j+1)} = \lambda_j-\lambda_i,\;\;\;(i,j)\in S,\\
q_{(i,j),(0,0)} &= \mu,\;\;\;(i,j)\neq (0,0).
\end{split}
\end{equation}
By imitating the proof in Lemma \ref{birthdeath}, it is easy to see that both $Y$ and $Z$ are skip-free Markov chains with birth rates $\lambda_n$ and death rates $\mu_n = q_{n0} = \mu$ \cite[Chapter III, Lemma 1.2 and Theorem 1.3]{liggett2005interacting}. According to the transition rates given in \eqref{transition2}, the distance between $Y$ and $Z$ remains the same or becomes closer after every single jump. This shows that
\begin{equation*}
\Pnum_m(Y^*_t\geq n)\leq \Pnum_0(Z^*_t\geq n-m),
\end{equation*}
which gives the desired result.
\end{proof}

Combining Corollary \ref{cor2} and the above lemma, we obtain the following theorem.
\begin{theorem}
Let $X$ be an upward skip-free Markov chain whose transition rates are given in \eqref{rates}. Then for any moderate function $F$, there exist two positive constants $c_F$ and $C_F$ such that for any stopping time $\tau$ of $X$,
\begin{equation*}
c_F\Enum F(\log_{\alpha+1}(\mu\tau+1)+1) \leq \Enum F(X^*_\tau+1) \leq C_F\Enum F(\log_{\alpha+1}(\mu\tau+1)+1).
\end{equation*}
\end{theorem}

The above theorem indicates that although the total birth and death rates of $X$ are exactly the same as those of the M/M/1 queue, the previous process does not exhibit phase transition. For any $\alpha>0$, $X^*_t$ on average behaves as $\log_{\alpha+1}t$ when $t$ is sufficiently large. This is because the large death jumps hinder $X$ from traveling too far.

\section{Proof of Theorem \ref{thm1}}\label{proof1}
The proof of Theorem \ref{thm1} is based on the discrete version of the Lenglart domination principle which is stated below (see \cite{peskir2001bounding} for the continuous version). Let $(\Omega,\F,\set{\F_t},\Pnum)$ be a filtered probability space. In the following lemma, the adapted processes and the stopping times are understood to be with respect to $\set{\F_t}$. Moreover, for any adapted process $Z$ and stopping time $\tau$, we shall use $Z^\tau$ to denote the stopped process $\{Z_{t\wedge\tau}:t\geq 0\}$.
\begin{lemma}\label{langlart}
Let $Z$ and $A$ be two c\`{a}dl\`{a}g adapted processes with $Z_0 = A_0 = 0$, where $Z$ is nonnegative and $A$ is increasing. Let $\set{x_n:n\geq 0}$ be an increasing sequences with $x_0 = 0$ and $x_\infty = \infty$. Let $h$ be an increasing step function on $\Rnum^+$ satisfying $h(0) = 0$ and
\begin{equation*}
h(x) = h(x_n),\;\;\;\textrm{if}\;x_n\leq x<x_{n+1}\;\textrm{for some $n\geq 0$}.
\end{equation*}
For any $n\geq 0$, let
\begin{equation*}
\tau_n  = \inf\set{t\geq 0:Z_t\geq x_n},\;\;\;\sigma_n  = \inf\set{t\geq 0:A_t\geq x_n}.
\end{equation*}
Assume that $Z_{\tau_n} = A_{\sigma_n} = x_n$ for any $n\geq 0$ and assume that $\Enum Z_\tau\leq\Enum A_\tau$ for any bounded stopping time $\tau$ such that $Z^\tau$ and $A^\tau$ are both bounded. Then for any bounded stopping time $\tau$,
\begin{equation*}
\Enum h(Z^*_\tau)\leq \Enum\tilde h(A_\tau),
\end{equation*}
where
\begin{equation}\label{tilde}
\tilde{h}(x) = 2h(x)+x\int_{(x,\infty)}\frac{1}{y}dh(y).
\end{equation}
\end{lemma}

\begin{proof}
For convenience, set $\Delta h_k = h(x_k)-h(x_{k-1})$ for any $k\geq 1$. Then
\begin{equation*}
\begin{split}
\Enum h(Z^*_\tau) &= \int_0^\infty\Pnum(Z^*_\tau\geq y)dh(y)
= \sum_{k=1}^\infty\Pnum(Z^*_\tau\geq x_k)\Delta h_k\\
&\leq \sum_{k=1}^\infty\left[\Pnum(Z^*_\tau\geq x_k,A_\tau<x_k)+\Pnum(A_\tau\geq x_k)\right]\Delta h_k.
\end{split}
\end{equation*}
Since $Z_{\tau_k} = A_{\sigma_k} = x_k$, it is easy to see that $\tau\wedge\tau_k\wedge\sigma_k$ is a bounded stopping time such that $Z^{\tau\wedge\tau_k\wedge\sigma_k}$ and $A^{\tau\wedge\tau_k\wedge\sigma_k}$ are both bounded. This suggests that
\begin{equation*}
\begin{split}
\Pnum(Z^*_\tau\geq x_k,A_\tau<x_k) &= \Pnum(\tau_k\leq\tau<\sigma_k)
\leq \Pnum(Z_{\tau\wedge\tau_k\wedge\sigma_k}\geq x_k)\\
&\leq \frac{1}{x_k}\Enum Z_{\tau\wedge\tau_k\wedge\sigma_k}
\leq \frac{1}{x_k}\Enum A_{\tau\wedge\tau_k\wedge\sigma_k}\\
&\leq \frac{1}{x_k}\Enum A_{\tau}I_{\set{A_\tau<x_k}}+\Pnum(A_\tau\geq x_k).
\end{split}
\end{equation*}
Thus we have
\begin{equation*}
\begin{split}
\Enum h(Z^*_\tau) &\leq \sum_{k=1}^\infty\left[\frac{1}{x_k}\Enum A_{\tau}1_{\set{A_\tau<x_k}}+2\Pnum(A_\tau\geq x_k)\right]\Delta h_k\\
&= \int_0^\infty\left[\frac{1}{y}\Enum A_{\tau}1_{\set{A_\tau<y}}+2\Pnum(A_\tau\geq y)\right]dh(y)
= \Enum\tilde h(A_\tau),
\end{split}
\end{equation*}
which gives the desired result.
\end{proof}

The following lemma plays a key role in the study of the maximal inequalities.
\begin{lemma}\label{martingale}
$\Enum f(X_\tau) = \Enum\tau$ for any bounded stopping time $\tau$ such that $X^\tau$ is bounded.
\end{lemma}

\begin{proof}
It has been proved in \cite{chen2014unified} that the function $f$ is the unique solution to the Poisson equation
\begin{equation*}
Qf = 1,\;\;\;f_0 = 0.
\end{equation*}
For any $n\geq 0$, let $\tau_n  = \inf\set{t\geq 0:X_t\geq n}$. Then $X^{\tau_n}$ is an upward skip-free Markov chain with absorbing state $n$, whose generator matrix is denoted by $Q_n$. Since $X^{\tau_n}$ is a Markov chain with finite state space, the process
\begin{equation*}
f(X^{\tau_n}_t)-f(X^{\tau_n}_0)-\int_0^tQ_nf(X^{\tau_n}_s)ds
\end{equation*}
is a martingale. It is easy to check that $Q_nf(X^{\tau_n}_s) = Qf(X_s)1_{\set{s<\tau_n}} = 1_{\set{s<\tau_n}}$. This shows that the process $f(X_{\tau_n\wedge t})-\tau_n\wedge t$ is a martingale. Thus for any bounded stopping time $\tau$ such that $X^\tau$ is bounded,
\begin{equation*}
\Enum f(X_{\tau_n\wedge\tau}) = \Enum\tau_n\wedge\tau.
\end{equation*}
By taking $n\rightarrow\infty$, the desired result follows from the dominated and monotonic convergence theorems.
\end{proof}

We are now in a position to prove Theorem \ref{thm1}.
\begin{proof}[Proof of Theorem \ref{thm1}]
Without loss of generality, we assume that $\tau$ is bounded. For any $p>0$, let $h_p$ be an increasing step function on $\Rnum^+$ defined as
\begin{equation*}
h_p(x) = n^p,\;\;\;\textrm{if}\;f_n\leq x<f_{n+1}\;\textrm{for some $n\geq 0$}.
\end{equation*}
Then $h_p(f_n) = n^p$ for any $n\geq 0$. In fact, the Peskir condition implies that there exists $C_p>0$ such that for any $n\geq 0$,
\begin{equation*}
f_n\sum_{k=n+1}^\infty\frac{h_p(f_k)-h_p(f_{k-1})}{f_k}\leq C_ph_p(f_n).
\end{equation*}
This shows that for any $n\geq 0$,
\begin{equation*}
\tilde h_p(f_n)\leq (C_p+2)h_p(f_n) = (C_p+2)n^p,
\end{equation*}
where $\tilde h_p$ is defined from $h_p$ in the same way that $\tilde{h}$ is defined from $h$ as in \eqref{tilde}. On one hand, if we choose $Z_t = f(X_t)$ and $A_t = t$, it is easy to check that all the assumptions in Lemma \ref{langlart} are satisfied. Thus we have
\begin{equation*}
\Enum (X^*_\tau)^p = \Enum h_p(f(X^*_\tau)) \leq \Enum\tilde h_p(\tau) \leq \Enum\tilde h_p(f(\rintg)) \leq (C_p+2)\Enum\rintg^p.
\end{equation*}
On the other hand, if we choose $Z_t = t$ and $A_t = f(X^*_t)$, it is also easy to check that all the assumptions in Lemma \ref{langlart} are satisfied. Thus we have
\begin{equation*}
\Enum\lintg^p = \Enum h_p(f(\lintg)) \leq \Enum h_p(\tau) \leq \Enum\tilde h_p(f(X^*_\tau)) \leq (C_p+2)\Enum(X^*_\tau)^p.
\end{equation*}
The above two inequalities give the desired result.
\end{proof}

\section{Proofs of Theorems \ref{thm2} and \ref{thm3}}\label{proof2}
We start with a simple fact.
\begin{lemma}\label{simple}
Let $F$ be a continuous increasing function on $\Rnum^+$ with $F(0) = 0$ and let $X$ be a nonnegative random variable. Then for any $\beta>0$,
\begin{equation*}
\int_0^\infty\Pnum(X\geq\lint{\beta x})dF(x) = \Enum F((\lint X+1)/\beta).
\end{equation*}
\end{lemma}

\begin{proof}
By the Fubini theorem, we have
\begin{equation*}
\begin{split}
\int_0^\infty\Pnum(X\geq\lint{\beta x})dF(x)
&= \sum_{k=0}^\infty\Pnum(X\geq k)[F((k+1)/\beta)-F(k/\beta)]\\
&= \Enum\sum_{k=0}^{\lint X}[F((k+1)/\beta)-F(k/\beta)] = \Enum F((\lint X+1)/\beta),
\end{split}
\end{equation*}
which gives the desired result.
\end{proof}

To proceed, we need the discrete version of the good $\lambda$ inequality which is stated below (see \cite[Page 164]{revuz1999continuous} for the continuous version).
\begin{lemma}\label{goodlambda}
Let $X$ and $Y$ be two nonnegative random variables. Let $\phi:\Rnum^+\rightarrow\Rnum^+$ be a function satisfying ${\phi(\delta)\rightarrow 0}$ as ${\delta\rightarrow 0}$. Assume that there exists $\beta>1$ such that the following good $\lambda$ inequality holds for any $\delta>0$ and sufficiently large integer $k$:
\begin{equation}\label{assumption}
\Pnum(X\geq\lint{\beta k},Y<\lint{\delta k}) \leq \phi(\delta)\Pnum(X\geq k).
\end{equation}
Then for any moderate function $F$, there exists $C>0$ depending on $\phi$, $\beta$, and $F$ such that
\begin{equation*}
\Enum F(\lint X+1)\leq C\Enum F(\lint Y+1).
\end{equation*}
\end{lemma}

\begin{proof}
Assume that \eqref{assumption} holds for any $\delta>0$ and $k\geq M$, where $M$ is a sufficiently large integer. We first prove that for any $0<\delta<1/M$ and $x>0$,
\begin{equation}\label{inequality}
\Pnum(X\geq\lint{\beta x},Y<\lint{\delta x}) \leq \phi(2\delta)\Pnum(X\geq\lint x).
\end{equation}
In order to prove this fact, we consider two different cases. If $0<x<M$, then $\lint{\delta x} = 0$. In this case, \eqref{inequality} holds trivially because its left side is zero. If $x\geq M$, there exists an integer $k\geq M$ such that $k\leq x<k+1$. In this case, we have $\lint{\beta x}\geq\lint{\beta k}$ and $\lint{\delta x}\leq\lint{\delta(k+1)}\leq \lint{2\delta k}$, which give rise to
\begin{equation*}
\begin{split}
\Pnum(X\geq\lint{\beta x},Y<\lint{\delta x})
&\leq \Pnum(X\geq\lint{\beta k},Y<\lint{2\delta k})\\
&\leq \phi(2\delta)\Pnum(X\geq k) = \phi(2\delta)\Pnum(X\geq\lint x).
\end{split}
\end{equation*}
Without loss of generality, we assume that $F$ is bounded. By Lemma \ref{simple}, for any $0<\delta<1/M$,
\begin{equation*}
\begin{split}
\Enum F((\lint X+1)/\beta)
&\leq \int_0^\infty\left[\Pnum(X\geq\lint{\beta x},Y<\lint{\delta x})+\Pnum(Y\geq\lint{\delta x})\right]dF(x)\\
&\leq \int_0^\infty\left[\phi(2\delta)\Pnum(X\geq\lint x)+\Pnum(Y\geq\lint{\delta x})\right]dF(x)\\
&\leq \phi(2\delta)\Enum F(\lint X+1)+\Enum F((\lint Y+1)/\delta).
\end{split}
\end{equation*}
Since $F$ is a moderate function and $\beta>1$, there exists $c>0$ such that $F(x/\beta)\geq cF(x)$ for any $x\geq 0$. This suggests that
\begin{equation*}
(c-\phi(2\delta))\Enum F(\lint X+1) \leq \Enum F((\lint Y+1)/\delta).
\end{equation*}
Since ${\phi(\delta)\rightarrow 0}$ as ${\delta\rightarrow 0}$, we can choose $0<\delta<1/M$ such that $\phi(2\delta)<c$. Since $F$ is a moderate function, there exists $C>0$ such that $F(x/\delta)\leq CF(x)$ for any $x\geq 0$. Thus we have
\begin{equation*}
\Enum F(\lint X+1) \leq \frac{C}{c-\phi(2\delta)}\Enum F(\lint Y+1),
\end{equation*}
which gives the desired result.
\end{proof}

We are now in a position to prove Theorem \ref{thm2}.
\begin{proof}[Proof of Theorem \ref{thm2}]
Since $X$ is controllable, there exist $\beta>1$ and $C,\gamma>0$ such that for any $t\geq 0$ and sufficiently large integer $k$,
\begin{equation*}
\Pnum_k(X^*_t\geq\lint{\beta k}) \leq C\Pnum_0(X^*_t\geq\lint{\gamma k}).
\end{equation*}
Let $\tau_k = \inf\set{t\geq 0:X_t\geq k}$. For any $\delta>0$ and integer $k\geq 0$, it is easy to see that
\begin{equation}\label{aim1}
\Pnum(X^*_\tau\geq\lint{\beta k},g(\tau)<\lint{\delta k})
\leq \Pnum(X^*_{s\vee\tau_k}\geq\lint{\beta k},\tau\geq\tau_k),
\end{equation}
where $s = f(\lint{\delta k})$. By the strong Markov property of $X$, we have
\begin{equation*}
\begin{split}
\Pnum(X^*_\tau\geq\lint{\beta k},g(\tau)<\lint{\delta k})
&\leq \Enum 1_{\{\tau\geq\tau_k\}}\Pnum(X^*_{s\vee\tau_ k}\geq\lint{\beta k}|\mathscr{F}_{\tau_k})\\
&\leq \Enum 1_{\{\tau\geq\tau_k\}}\Pnum_{X_{\tau_k}}(X^*_{s\vee\tau_k-\tau_k}\geq\lint{\beta k})\\
&\leq \Pnum_k(X^*_s\geq\lint{\beta k})\Pnum(X_\tau^*\geq k).
\end{split}
\end{equation*}
By Theorem \ref{thm1}, there exists $M\geq 1$ such that for any integer $k\geq M$,
\begin{equation*}
\Pnum_k(X^*_s\geq\lint{\beta k}) \lesssim \Pnum_0(X^*_s\geq\lint{\gamma k})
\leq \frac{\Enum(X^*_s)^p}{\lint{\gamma k}^p}
\lesssim \frac{\lint{\delta k}^p}{\lint{\gamma k}^p}
\leq \frac{(\delta k)^p}{(\gamma k-1)^p} \lesssim \delta^p.
\end{equation*}
Thus there exists $c_p>0$ such that for any $\delta>0$ and integer $k\geq M$,
\begin{equation*}
\Pnum(X^*_\tau\geq\lint{\beta k},g(\tau)<\lint{\delta k}) \leq c_p\delta^p\Pnum(X^*_\tau\geq k).
\end{equation*}
By Lemma \ref{goodlambda}, there exists $C_F>0$ such that
\begin{equation*}
\Enum F(X^*_\tau+1) \leq C_F\Enum F(\lint{g(\tau)}+1) \leq C_F\Enum F(g(\tau)+1),
\end{equation*}
which gives the desired result.
\end{proof}

We are now in a position to prove Theorem \ref{thm3}.
\begin{proof}[Proof of Theorem \ref{thm3}]
For any $\delta>0$ and integer $k\geq 0$, it is easy to see that
\begin{equation}\label{aim2}
\Pnum(g(\tau)\geq\lint{\beta k},X^*_\tau<\lint{\delta k})
\leq \Pnum(\tau\geq r,X^*_s<\lint{\delta k}),
\end{equation}
where $r = f(k)$ and $s = f(\lint{\beta k})$. By the Markov property of $X$, we have
\begin{equation*}
\begin{split}
\Pnum(g(\tau)\geq\lint{\beta k},X^*_\tau<\lint{\delta k})
&\leq \Enum 1_{\{\tau\geq r\}}\Pnum(X^*_s<\lint{\delta k}|\mathscr{F}_r)
\leq \Enum 1_{\{\tau\geq r\}}\Pnum_{X_r}(X^*_{s-r}<\lint{\delta k})\\
&\leq \Pnum(X^*_{s-r}<\lint{\delta k})\Pnum(g(\tau)\geq k).
\end{split}
\end{equation*}
Let $\tau_{\delta k} = \inf\set{t\geq 0:X_t\geq\lint{\delta k}}$. It follows from Lemma \ref{martingale} that $f(X_{\tau_{\delta k}\wedge t})-\tau_{\delta k}\wedge t$ is a martingale. This shows that for any $t\geq 0$,
\begin{equation*}
\Enum \tau_{\delta k}\wedge t = \Enum f(X_{\tau_{\delta k}\wedge t})\leq f(\lint{\delta k}).
\end{equation*}
Thus we have
\begin{equation*}
\begin{split}
\Pnum(X^*_{s-r}<\lint{\delta k}) = \Pnum(\tau_{\delta k}>s-r)
\leq \frac{1}{s-r}\Enum\tau_{\delta k}
\leq \frac{f(\lint{\delta k})}{f(\lint{\beta k})-f(k)}.
\end{split}
\end{equation*}
Thus for any $\delta>0$ and integer $k\geq M$,
\begin{equation*}
\Pnum(g(\tau)\geq\lint{\beta k},X^*_\tau<\lint{\delta k})
\leq \sup_{k\geq M}\frac{f(\lint{\delta k})}{f(\lint{\beta k})-f(k)}\Pnum(g(\tau)\geq k).
\end{equation*}
Thus Lemma \ref{goodlambda}, together with \eqref{jia}, shows that there exists $k_F>0$ such that
\begin{equation*}
\Enum F(X^*_\tau+1) \geq k_F\Enum F(\lint{g(\tau)}+1).
\end{equation*}
It is easy to see that $x+1 \leq 2(\lint x+1)$ for any $x\geq 0$. Since $F$ is a moderate function, there exists $c>0$ such that $F(x/2)\geq cF(x)$ for any $x\geq 0$. Thus we finally obtain that
\begin{equation*}
\Enum F(X^*_\tau+1) \geq k_F\Enum F((g(\tau)+1)/2)\geq ck_F\Enum F(g(\tau)+1),
\end{equation*}
which gives the desired result.
\end{proof}

\section*{Acknowledgements}
I am grateful to G.-H. Zhao, X.-F. Xue, and X. Chen for stimulating discussions and grateful to the anonymous reviewers for their valuable suggestions which help me greatly in improving the quality of this paper.

\setlength{\bibsep}{8pt}
\small\bibliographystyle{nature}

\end{document}